\documentclass[11pt,a4paper]{article}

\usepackage{amsmath,amssymb,latexsym}
\usepackage{mathtools}
\usepackage{mathrsfs}
\usepackage{graphicx}
\usepackage{epsfig}
\usepackage{longtable}
\usepackage{amsthm}
\usepackage{xcolor}
\usepackage{hyperref}
\usepackage{tikz}
\usepackage[all]{xy}
\usepackage{booktabs}

\usepackage{lineno}

\newtheorem{theorem}{Theorem}
\newtheorem{lemma}[theorem]{Lemma}
\newtheorem{corollary}[theorem]{Corollary}
\newtheorem{proposition}[theorem]{Proposition}
\newtheorem{definition}{Definition}

\newtheorem{remark}{Remark}

\begin{document}

\title{ \Large\bf New Baire category results for stochastic orders on bivariate copulas}

\author{María del Rosario Rodríguez-Griñolo$^{\rm a}$, Manuel Úbeda-Flores$^{\rm b,}$\footnote{Corresponding author}\medskip
\\
\small{$^{\rm a}$Department of Economics, Quantitative Methods and Economic History, Pablo de Olavide University,}\\
\small{41013 Seville, Spain}\\
\small{\texttt{mrrodgri@upo.es}}\\
\small{$^{\rm b}$Department of Mathematics, University of Almería, 04120 Almería, Spain}\\
\small{\texttt{mubeda@ual.es}}
}
\maketitle

\begin{abstract}
In the sense of Baire categories, we prove that the set of pairs of bivariate copulas that are comparable---in either direction---under the increasing convex order is nowhere dense in the space of all pairs of bivariate copulas equipped with the uniform metric. As a consequence, a topologically generic pair of bivariate copulas is not comparable under this order. We further extend the Baire-category programme to two additional stochastic orders on the space of bivariate copulas: the bivariate convex order and the stop-loss order on the sum of the components. For each of these orders, we establish that the set of comparable pairs is closed and nowhere dense, and we show that a topologically generic pair of bivariate copulas is simultaneously incomparable in all three orders. These results complement those obtained in~[F. Durante, J. Fern\'andez-S\'anchez, C. Ignazzi (2022). Baire category results for stochastic orders. \emph{Rev. Real Acad. Cienc. Exactas Fis. Nat. Ser. A-Mat.} {\bf 116}, article 188] for the lower orthant order on copulas.
\end{abstract}

 

\section{Introduction}
\label{sec:introduction}

Stochastic orders provide a systematic framework for comparing probability distributions, and constitute one of the most active and fruitful research areas in probability theory and mathematical statistics. The monographs ~\cite{MuellerStoyan2002,ShakedShanthikumar2007} offer comprehensive surveys of the subject; we refer  to these works for background. Among the most studied partial orders on spaces of distribution functions are the usual stochastic order $\le_{\mathrm{st}}$, the increasing convex order $\le_{\mathrm{icx}}$, the mean residual life order $\le_{\mathrm{mrl}}$, and---in the
multivariate setting---the lower orthant order
$\le_{\mathrm{lo}}$ and the supermodular order
$\le_{\mathrm{sm}}$.
 
Given a partial order $\le_{*}$ and two distribution functions $F$ and $G$, a natural and practically relevant question is whether $F$
and $G$ are comparable, i.e., whether $F\le_{*}G$ or $G\le_{*}F$ holds. This question has direct relevance in actuarial science and
quantitative risk management, where $F$ and $G$ typically model the distributions of losses or risks
(see~\cite{MuellerStoyan2002,ShakedShanthikumar2007}). Despite the theoretical importance of stochastic orders, the question of how many pairs of distribution functions are actually
comparable---in a precise topological or measure-theoretic sense---has received limited attention in the literature.
 
A topological approach to this question, based on the notion of Baire categories, was recently developed in
\cite{DuranteFernandezIgnazzi2022}.
Recall that, in a topological space, a set is meager (or of first category) if it can be expressed as a countable union of nowhere dense sets, and co-meager (or residual) if its
complement is meager~\cite{Oxtoby1980}.
In a complete metric space, co-meager sets are large in the topological sense: by the Baire Category Theorem, they are of second category, and any countable intersection of co-meager sets in a complete metric space remains co-meager. Elements belonging to a co-meager set are called generic (or typical).
 
In~\cite{DuranteFernandezIgnazzi2022}, the authors proved that the set of all pairs of univariate distribution functions (defined on a bounded subset of $\mathbb{R}$) that are
comparable in the usual stochastic order $\le_{\mathrm{st}}$, the increasing convex order $\le_{\mathrm{icx}}$, or the mean residual life order $\le_{\mathrm{mrl}}$ is nowhere dense (and, {\it a fortiori}, meager) in the space $\mathcal{F}\times\mathcal{F}$ equipped with the Lévy metric. They also established an analogous result in the class of bivariate copulas $\mathcal{C}$: the set of $\le_{\mathrm{lo}}$-comparable pairs of copulas is nowhere dense in
$(\mathcal{C}\times\mathcal{C}, d\times d)$, where $d$ denotes the uniform metric.
 
The present paper addresses a natural and open question left by
\cite{DuranteFernandezIgnazzi2022}: \emph{can similar Baire-category results be established for the increasing convex order directly on the space of bivariate copulas?} This order is richer and more difficult to handle than $\le_{\mathrm{lo}}$, since it involves the full probability measure induced by the copula---not merely its pointwise values---and it captures genuinely different aspects of stochastic dominance in the multivariate setting.
 
We prove that the set
\[
\mathcal{Q}=\bigl\{(A,B)\in\mathcal{C}\times\mathcal{C}:
(A\le_{\mathrm{icx}}B)\vee(B\le_{\mathrm{icx}}A)\bigr\}
\]
(also denoted as $\mathcal{P}_{icx}$) is nowhere dense in $(\mathcal{C}\times\mathcal{C},d\times d)$
(Theorem~\ref{thm:icx_nowhere_dense}). The proof strategy follows the general Baire-category framework of~\cite{DuranteFernandezIgnazzi2022}, reusing the geometric construction of~\cite[Theorem~6]{DuranteFernandezIgnazzi2022}, but requiring substantially new arguments to handle the non-pointwise nature of $\le_{\rm icx}$. The key difficulty with $\le_{\mathrm{icx}}$ on copulas, compared to $\le_{\mathrm{lo}}$, is twofold. The order $\le_{\mathrm{icx}}$ cannot be verified pointwise: it involves integration against the entire class $\Phi_{\mathrm{icx}}$ of increasing convex functions. The class of increasing convex functions on $[0,1]^2$ does not separate copulas as cleanly as monotone functions do (since many elements of $\Phi_{\mathrm{icx}}$ integrate to the same value for all copulas with uniform marginals).

Beyond the increasing convex order, we establish analogous results for two further stochastic orders on $\mathcal{C}\times\mathcal{C}$
(Section~\ref{sec:otherorders}). For the \emph{bivariate convex order} $\le_{\mathrm{\rm cx}}$, we establish the same conclusion (Proposition~\ref{prop:cx}). For the \emph{stop-loss order on the sum of the components}
$\le_{\mathrm{sl}}$ —that is, the (univariate) increasing convex order applied to the marginal sums $S_A=X+Y$ and $S_B=X'+Y'$ associated,
respectively, to $(X,Y)\sim A$ and $(X',Y')\sim B$, where ``$\sim$'' denotes ``has joint distribution function''— we again prove that the
set of comparable pairs is nowhere dense
(Proposition~\ref{prop:lorenz}). Combining these results, Theorem~\ref{thm:omnibus34} shows that a topologically generic pair of bivariate copulas is simultaneously incomparable in all three orders $\le_{\mathrm{icx}}, \le_{\mathrm{\rm cx}}, \le_{\mathrm{\rm sl}}$.

The paper is organised as follows.
Section~\ref{sec:preliminaries} collects the necessary background on copulas, stochastic
orders, and Baire categories.
Section~\ref{sec:icx_copulas} contains the main results for the increasing convex order:
the closedness of $\le_{\mathrm{icx}}$ in $\mathcal{C}$, and the nowhere-denseness of $\mathcal{Q}$. Section~\ref{sec:otherorders} extends the Baire-category programme to the convex and stop-loss orders on bivariate copulas. Section~\ref{sec:conc} is devoted to conclusions.
 
\section{Preliminaries} \label{sec:preliminaries}
 
We collect in this section the definitions and background results needed throughout the paper. The section is divided into three parts: copula theory
(Section~\ref{sec:copulas}), stochastic orders
(Section~\ref{sec:stoc}), and Baire category theory (Section~\ref{sec:bairecat}).
Standard references for copula theory
are~\cite{DuranteSempi2016,Nelsen2006}.
For stochastic orders we refer to~\cite{MuellerStoyan2002,ShakedShanthikumar2007}. For Baire categories the standard reference is~\cite{Oxtoby1980}.
 
\subsection{Copulas}\label{sec:copulas}

 We begin by recalling the definition and basic properties of bivariate copulas. A (bivariate) \emph{copula} is a function $C: [0, 1]^2 \to [0, 1]$ that satisfies:
\begin{itemize}
\item[(C1)] \emph{Boundary conditions}: For every $t \in [0, 1]$, $C(t, 0) = C(0, t) = 0$ and $C(t, 1) = C(1, t) = t$.
\item[(C2)] \emph{The $2$-increasing property}: For any $u, v, u', v' \in [0, 1]$ such that $u \le u'$ and $v \le v'$, $
C(u', v') - C(u, v') - C(u', v) + C(u, v) \geq 0$.
\end{itemize}

According to Sklar's theorem \cite{Sklar1959}, for any bivariate random vector $(X,Y)$ with joint distribution function $H$ and marginals $F$ and $G$, there exists a copula $C$ (which is uniquely determined on ${\rm Range}\,
F\,\times{\rm Range}\,G\,$) such that $
H(x,y) = C(F(x),G(y))$ for all $(x,y)\in[-\infty,+\infty]^2$; moreover, if the marginals are continuous, then the copula is unique (for a complete proof, see \cite{Ubeda2017}). This copula captures the entire dependence structure between $X$ and $Y$, independently of the marginal effects. Essentially, $C$ corresponds to a bivariate distribution function on $[0, 1]^2$ with uniform marginals.

The \emph{Fréchet class} $\mathcal{F}(F_{1},F_{2})$ denotes the set of all bivariate distribution functions with fixed marginals $F_{1}$ and $F_{2}$; by Sklar's theorem it is in bijection with $\mathcal{C}$.
 
The class $\mathcal{C}$ is equipped with the \emph{uniform metric}
\begin{equation*}\label{eq:uniform_metric}
  d(A,B) \;:=\; \sup_{(x,y)\in[0,1]^{2}}|A(x,y)-B(x,y)|,
\end{equation*}
which metrizes uniform convergence on $[0,1]^{2}$. Within $\mathcal{C}$, uniform convergence is equivalent to weak convergence of the associated probability measures
(see~\cite{DuranteSempi2016,Sempi2004}). We equip $\mathcal{C}\times\mathcal{C}$ with the maximum product metric
\[
  (d\times d)\bigl((A,B),(A',B')\bigr) := \max\{d(A,A'),\, d(B,B')\}.
\]
Since $(\mathcal{C},d)$ is complete and compact \cite{DuranteSempi2016}, so is $(\mathcal{C}\times\mathcal{C},d\times d)$,
and the open balls of $d\times d$ are exactly the products $B_d(A,\varepsilon)\times B_d(B,\varepsilon)$.
 
Every copula $C\in\mathcal{C}$ satisfies the pointwise bounds
\begin{equation*}\label{eq:frechet_bounds}
  W(x,y) \;\le\; C(x,y) \;\le\; M(x,y)
  \quad\text{for every }(x,y)\in[0,1]^{2},
\end{equation*}
where the \emph{Fréchet-Hoeffding bounds} are
$W(x,y):=\max(x+y-1,0)$ and $M(x,y):=\min(x,y)$.
Both $W$ and $M$ are themselves copulas: $M$ is the copula of the comonotonic pair $(U,U)$ and $W$ is the copula of the counter-monotonic pair $(U,1-U)$ for $U\sim\mathrm{Uniform}(0,1)$.
 
Every copula satisfies a \emph{1-Lipschitz condition} with respect
to the $\ell^{1}$ norm:
\begin{equation*}\label{eq:Lipschitz}
  |C(x_{1},y_{1})-C(x_{2},y_{2})|
  \;\le\; |x_{1}-x_{2}|+|y_{1}-y_{2}|
\end{equation*}
for every $(x_{1},y_{1}),(x_{2},y_{2})\in[0,1]^{2}$.
 
The \emph{diagonal section} of a copula $C$ is the function $\delta_{C}\colon[0,1]\to[0,1]$ defined by $\delta_{C}(t):=C(t,t)$. It satisfies $\delta_{C}(t)\le t$ for every $t\in[0,1]$.
 
Ordinal sums of copulas provide a flexible construction tool
\cite{DuranteSempi2016,DuranteKlementSamingerSempi2022}. Given copulas $C_1,\ldots,C_n$ and a partition of $[0,1]$ into consecutive subintervals $[a_0,a_1],[a_1,a_2],\ldots,[a_{n-1},a_n]$ with $0=a_0<a_1<\cdots<a_n=1$, the \emph{ordinal sum} (denoted $\oplus$) is the copula defined by
\begin{equation*}\label{eq:ordinalsum}
C(x,y) \;=\; a_{k-1} + (a_k-a_{k-1})\,C_k\!\left(
   \frac{x-a_{k-1}}{a_k-a_{k-1}},\;\frac{y-a_{k-1}}{a_k-a_{k-1}}
\right)
\quad\text{if }(x,y)\in(a_{k-1},a_k)^2,
\end{equation*}
and $C(x,y)=M(x,y)$ at all other points of $[0,1]^2$. We write $(C_1,[a_0,a_1])\oplus(C_2,[a_1,a_2])\oplus\cdots
        \oplus(C_n,[a_{n-1},a_n])$, where each summand
$(C_k,[a_{k-1},a_k])$ specifies that the copula $C_k$ is rescaled affinely to the subsquare $[a_{k-1},a_k]^2$ and then glued to the comonotonic copula $M$ outside those subsquares. The value $M(x,y)$ outside the diagonal blocks $[a_{k-1},a_k]^2$ ensures that 
$C$ has the correct boundary behaviour and uniform marginals; the 2-increasing 
property throughout $[0,1]^2$ follows from the general theory of ordinal 
sums~\cite{DuranteKlementSamingerSempi2022}. Note that $\delta_C(a_k)=a_k$ for every $k$, i.e., the ordinal sum touches the main diagonal at each breakpoint $a_k$.

\subsection{Stochastic orders}\label{sec:stoc}

We introduce the stochastic orders that appear in the paper. We first recall the classical univariate orders on the space of distribution functions on $[0,1]$, which provide context for the bivariate setting. The multivariate orders directly studied here are defined afterwards.
 
Let $\mathcal{F}$ denote the class of all right-continuous distribution functions $F\colon[0,1]\to[0,1]$ satisfying $F(0^{+})\ge 0$ and $F(1)=1$. For $F\in\mathcal{F}$, we write $\overline{F}:=1-F$ for the associated survival function, and we denote by
$$m_{F}(x):=\frac{1}{\overline{F}(x)}\int_{x}^{1}\overline{F}(t)\,{\rm d}t$$
the \emph{mean residual life function} for  $x\in[0,1]$ with $\overline{F}(x)>0$, and by convention $m_F(x):=0$ when $\overline{F}(x)=0$.
 
\begin{definition}
\label{def:univariate_orders}
Let $F,G\in\mathcal{F}$.
\begin{enumerate}
  \item $F\le_{\mathrm{st}} G$ \emph{(usual stochastic order)} if
        $\overline{F}(x)\le\overline{G}(x)$ for every $x\in[0,1]$.
  \item $F\le_{\mathrm{icx}} G$ \emph{(increasing convex order)} if
$$\int_{0}^1\varphi(x)\,{\rm d}F(x)\le\int_{0}^1\varphi(x)\,{\rm d}G(x)$$
        for every increasing convex function
        $\varphi\colon[0,1]\to\mathbb{R}$.
        Equivalently,
        $F\le_{\mathrm{icx}} G$ if, and only if,
\begin{equation*}\label{eq:icx_univariate_char}
          \int_{t}^{1}\overline{F}(x)\,{\rm d}x\le
          \int_{t}^{1}\overline{G}(x)\,{\rm d}x
          \quad\text{for every }t\in[0,1].
        \end{equation*}
  \item $F\le_{\mathrm{mrl}} G$ \emph{(mean residual life order)} if
        $m_{F}(x)\le m_{G}(x)$ for every $x\in[0,1]$ such that
        $\overline{F}(x)>0$ and $\overline{G}(x)>0$.
\end{enumerate}
\end{definition}
 
The following implications are standard:
$F\le_{\mathrm{st}} G\Rightarrow F\le_{\mathrm{icx}} G$ and
$F\le_{\mathrm{mrl}} G\Rightarrow F\le_{\mathrm{icx}} G$
(see~\cite{DuranteFernandezIgnazzi2022,ShakedShanthikumar2007}).
In general, $\le_{\mathrm{st}}$ and $\le_{\mathrm{mrl}}$ are not
comparable~\cite{ShakedShanthikumar2007}.
 
The next bivariate stochastic orders for copulas are the main objects of study in this paper.
 
\begin{definition}\label{def:multivariate_orders}
Let $A,B\in\mathcal{C}$.
\begin{enumerate}
  \item $A\le_{\mathrm{lo}} B$ \emph{(lower orthant order)} if
        $B(x,y)\le A(x,y)$ for every $(x,y)\in[0,1]^{2}$.
  \item $A\le_{\mathrm{sm}} B$ \emph{(supermodular order)} if
$$\int_{[0,1]^{2}}\varphi(x,y)\,{\rm d}A(x,y)\le\int_{[0,1]^{2}}\varphi(x,y)\,{\rm d}B(x,y)$$
        for every supermodular function
        $\varphi\colon[0,1]^{2}\to\mathbb{R}$, i.e., every
        $\varphi$ satisfying
$\varphi(x_{1},y_{1})+\varphi(x_{2},y_{2})\ge
        \varphi(\min\{x_{1},x_{2}\},\min\{y_{1}, y_{2}\})+
        \varphi(\max\{x_{1},x_{2}\},\max\{y_{1},y_{2}\})$
        for all $x_{1},x_{2},y_{1},y_{2}\in[0,1]$.
  \item $A\le_{\mathrm{icx}} B$ \emph{(increasing convex order)} if
$$\int_{[0,1]^{2}}\varphi(x,y)\,{\rm d}A(x,y)\le\int_{[0,1]^{2}}\varphi(x,y)\,{\rm d}B(x,y)$$
        for every function $\varphi\in\Phi_{\mathrm{icx}}$, where
        $$\Phi_{\mathrm{icx}}:=\bigl\{\varphi:[0,1]^{2}\to\mathbb{R} : \varphi
\text{ is increasing, convex and finite-valued}\bigr\}.$$
        Here, \emph{increasing} means non-decreasing in each argument, and \emph{convex} is understood in the usual sense for functions
        on $\mathbb{R}^{2}$.
\item $A\leq_{\mathrm{sl}} B$ (\emph{stop-loss order on sums}) if
\[
\mathbb{E}\bigl[(S_A-t)^+\bigr]\le \mathbb{E}\bigl[(S_B-t)^+\bigr]
\quad\text{for every } t\in\mathbb{R},
\]
where $S_A:=X+Y$ under $(X,Y)\sim A$,  $S_B:=X'+Y'$ under $(X',Y')\sim B$ and $z^+=\max(z,0).$
\item $A\le_{\rm cx}B$ (\emph{convex order}) if
$$\int_{[0,1]^2}\varphi(x,y)\,{\rm d}A(x,y)\le\int_{[0,1]^2}\varphi(x,y)\,{\rm d}B(x,y)$$ for every convex $\varphi:[0,1]^2\to\mathbb R$ (no monotonicity required).
\end{enumerate}
\end{definition}

\begin{remark}\label{rem:order_implications}
For bivariate copulas, $A\le_{\mathrm{sm}} B$ is equivalent to
$B\le_{\mathrm{lo}} A$~\cite[Theorem~3.8.2]{MuellerStoyan2002}. Moreover, $A\le_{\mathrm{sm}} B$ implies the
stop-loss order $A\le_{\mathrm{sl}} B$, since
$(x,y)\mapsto (x+y-t)^+$ is supermodular for every $t\in\mathbb{R}$. Note, however, that $\le_{\mathrm{sm}}$ does not in general imply the bivariate increasing convex order $\le_{\mathrm{icx}}$: a componentwise increasing, jointly convex function on $[0,1]^2$ need not be supermodular ---consider, e.g., $f(x,y)=5x+5y+(x^2-xy+y^2)/2$.
Note also that $\le_{\rm sm}$ and $\le_{\rm icx}$ are incomparable in general: neither implies the other~\cite{MuellerStoyan2002}. Moreover, $\le_{\rm lo}$ and $\le_{\rm icx}$ are incomparable in general as well: a pointwise domination of copulas does not imply domination in the increasing convex order, nor vice versa.
\end{remark}

\begin{remark}\label{rem:sl}
By definition, $A\leq_{\mathrm{sl}} B$ is exactly the increasing convex order applied to the one-dimensional random variables $S_A$ and $S_B$ (see~\cite[Section~4.A.1]{ShakedShanthikumar2007}). Because every bivariate copula has uniform marginals on $[0,1]$, one has $\mathbb{E}[S_A]=\mathbb{E}[S_B]=1$ independently of $A,B$; in this equal-means case, the stop-loss and the (univariate) convex orders applied to $S_A,S_B$ all coincide (see~\cite[Theorem 3.A.1]{ShakedShanthikumar2007}; and also \cite[\S\,3.2]{ArnoldSarabia2018} for the Lorenz-order formulation), which justifies referring to $\leq_{\mathrm{sl}}$ also as a stop-loss-type order in the present setting. Equivalently, $A\leq_{\mathrm{sl}} B$ if, and only if,
\[
\int_t^{\infty}\overline{F}_{S_A}(s)\,{\rm d}s\;\le\; \int_t^{\infty}\overline{F}_{S_B}(s)\,{\rm d}s\quad\text{for all } t\ge 0.
\]
(see~\cite{DenuitDhaeneGoovaertsKaas2005} for details).
\end{remark}

\begin{remark}
For given copulas $A$ and $B$, every affine function $\varphi(x,y)=ax+by+c$ satisfies
\begin{align*}
\int_{[0,1]^2}\varphi(x,y)\,{\rm d}(B-A)(x,y)
= a\int_0^1 x\,{\rm d}(B_X-A_X)(x)
 + b\int_0^1 y\,{\rm d}(B_Y-A_Y)(y) = 0,
\end{align*}
where $A_X, B_X$ (respectively, $A_Y, B_Y$) denote the marginal distributions of 
$A$ and $B$ on the first (respectively, second) coordinate; since $A,B\in\mathcal{C}$ 
both have standard uniform marginals, $A_X=B_X=\mathcal{U}(0,1)$ and 
$A_Y=B_Y=\mathcal{U}(0,1)$, so both integrals vanish. Hence $A\le_{\rm cx}B$ depends only on the equivalence class of $\varphi$ modulo 
affine functions; in particular, it suffices to test against any countable, uniformly dense subfamily of the cone of continuous convex functions on $[0,1]^2$.
\end{remark}

\subsection{Baire categories}\label{sec:bairecat}
 
We recall the relevant notions from Baire category theory; the standard reference is~\cite{Oxtoby1980}. Given a topological space $(X,\tau)$, a subset $A\subseteq X$ is called \emph{nowhere dense} if its closure $\overline{A}$ has empty interior, i.e., $\overset{\circ}{\overline{A}}=\emptyset$. A set expressible as a countable union of nowhere dense sets is called \emph{meager} (or \emph{of first category}). The complement of a meager set is called \emph{co-meager} (or \emph{residual}), and its elements are termed \emph{typical}.
 
The cornerstone of this theory is the following result.
 
\begin{theorem}[{\bf Baire Category}]
\label{thm:baire}
Every complete metric space is a Baire space, i.e., every non-empty open set is of second category. Equivalently, in a complete metric space, no open set is meager. In particular, every co-meager subset of a complete metric space is of second category.
\end{theorem}
 
An important observation is that, in Baire spaces, a co-meager set $A\subseteq X$ is necessarily of second category: if $A$ were meager, then $X = A\cup A^c$ would itself be meager (as the union of two meager sets, since $A^c$ is meager by definition of co-meager), contradicting the Baire property.
 
Since the product of complete metric spaces is complete, and $(\mathcal{C},d)$ is complete~(see \cite{DuranteSempi2016}), the product space $(\mathcal{C}\times\mathcal{C}, d\times d)$ is a complete metric space, hence a Baire space by Theorem~\ref{thm:baire}.

Let $\mathcal{F}$ be the class of (right-continuous) distribution functions $F$ with support on $[0,1]$ such that $F$ is non-decreasing, and $F(1) = 1$; and set $\mathcal{F}^2 := \mathcal{F} \times \mathcal{F}$ equipped with the topology induced by the metric $d_L \times d_L$, where $d_L$ denotes the Lévy metric (see \cite{Bill99}). We record here the results of~\cite{DuranteFernandezIgnazzi2022}
that are directly related to the present work. 
 
\begin{theorem}\label{thm:DFI_univariate}
For every $\leq_*\in\{\leq_{\mathrm{st}},\leq_{\mathrm{icx}},\leq_{\mathrm{mrl}}\}$, the set $\mathcal{F}^2_*:=\{(F,G)\in\mathcal{F}^2 : F\leq_* G \text{ or } G\leq_* F\}$ is nowhere dense in $(\mathcal{F}^2,d_L\times d_L)$, and its complement is co-meager.
\end{theorem}
 
\begin{theorem}\label{thm:DFI_lo}
The set $\mathcal{P}:=\{(A,B)\in\mathcal{C}\times\mathcal{C}:
A\le_{\mathrm{lo}}B\text{ or }B\le_{\mathrm{lo}}A\}$
is nowhere dense in $(\mathcal{C}\times\mathcal{C},d\times d)$,
and its complement is co-meager. Consequently, two elements of the same Fréchet class
$\mathcal{F}(F_{1},F_{2})$ of continuous distribution functions are typically not comparable in the lower orthant order.
\end{theorem}
 
Theorem~\ref{thm:DFI_univariate} covers $\le_{\mathrm{icx}}$ for univariate distributions on $[0,1]$. Theorem~\ref{thm:DFI_lo} covers the lower orthant order on copulas. The present paper fills the remaining gap by treating
the increasing convex order $\leq_{\mathrm{icx}}$, the bivariate convex order $\leq_{\mathrm{\rm cx}}$, and the stop-loss order $\leq_{\mathrm{\rm sl}}$ on the sum of components directly on bivariate copulas, a setting left open in \cite{DuranteFernandezIgnazzi2022}.

\section{Baire category results for the increasing convex order}
\label{sec:icx_copulas}

In this section we investigate the topological size of the $\le_{\mathrm{icx}}$-comparable pairs of copulas from the perspective of Baire category theory. We first establish the closedness of $\mathcal{Q}$ in
$(\mathcal{C}\times\mathcal{C},d\times d)$
(Lemma~\ref{lem:icx_closed} and Corollary~\ref{cor:Q_closed}), which is a prerequisite for the nowhere-denseness argument.
We then prove our main result (Theorem~\ref{thm:icx_nowhere_dense}).

\subsection{Closedness of the increasing convex order in \texorpdfstring{$\mathcal{C}$}{}}

Before proving the closedness result, we recall the following standard theorem, whose proof can be found, e.g., in \cite{Dudley89}.

\begin{theorem}[{\bf Portmanteau}]\label{th:port}
    For probability measures $\mu_n$ ($n\ge 1$) and $\mu$ on a metric space $(E,\mathcal{B}(E))$ the following statements are equivalent:
    \begin{itemize}
        \item [i)] $\mu_n$ converges weakly to $\mu$ (denoted by $\mu_n \xrightarrow{w} \mu$).
        \item [ii)] For all measurable open set $S$, $\displaystyle\liminf_{n\to\infty}\mu_n(S)\geq\mu(S)$.
        \item [iii)] For all measurable closed set $F$, $\displaystyle\limsup_{n\to\infty}\mu_n(F)\leq\mu(F)$.
        \item [iv)] $\displaystyle\int g\, {\rm d}\mu_n\longrightarrow\int g\,{\rm d}\mu$, for every bounded measurable function $g:E\longrightarrow\mathbb{R}$ with $\mu(\mathrm{Disc}(g))=0$, where $\mathrm{Disc}(g)$ 
denotes the set of discontinuity points of $g$.
    \end{itemize}
\end{theorem}

With this tool at hand, we can now establish the closedness of the set of $\le_{\rm icx}$-comparable pairs.

\begin{lemma}\label{lem:icx_closed}
The set
\[
  \mathcal{Q}_{1}
  \;:=\;
  \bigl\{(A,B)\in\mathcal{C}\times\mathcal{C}
  :A\le_{\mathrm{icx}} B\bigr\}
\]
is closed in $(\mathcal{C}\times\mathcal{C},d\times d)$.
\end{lemma}
 
\begin{proof}
Let $\{(A_{n},B_{n})\}_{n\ge 1}\subseteq\mathcal{Q}_{1}$ be a sequence
with $d(A_{n},A)\to 0$ and $d(B_{n},B)\to 0$, for some
$A,B\in\mathcal{C}$.
We show $A\le_{\mathrm{icx}} B$.
 
Fix any $\varphi\in\Phi_{\rm icx}$.
Since $(A_{n},B_{n})\in\mathcal{Q}_{1}$,
\begin{equation*}\label{eq:ineq_n}
  \int_{[0,1]^{2}}\varphi(x,y)\,{\rm d}A_{n}(x,y)
  \;\le\; \int_{[0,1]^{2}}\varphi(x,y)\,{\rm d}B_{n}(x,y)
  \quad\text{for all }n\ge 1.
\end{equation*}
Every $\varphi\in\Phi_{\mathrm{icx}}$ is continuous on the open square $(0,1)^2$ ---a convex function on $\mathbb{R}^n$ is continuous on the interior of any convex set on which it is finite: see e.g.~\cite[Theorem~10.1]{Rockafellar1970}. On each of the four edges of $\partial[0,1]^2$ ---the boundary of the unit square---, $\varphi$ restricts to a convex function of one variable on a compact interval, hence is bounded; combined with continuity in the interior this yields a global bound $|\varphi|\le K_\varphi$ for some $K_\varphi<\infty$. The set $\mathrm{Disc}(\varphi)$ is 
contained in $\partial[0,1]^2$. For every copula $C \in \mathcal{C}$, the marginal 
of $\mu_C$ on each coordinate is Lebesgue measure on $[0,1]$, so 
$\mu_C(\{0\} \times [0,1]) = \mathbb{P}(X=0) = 0$, $\mu_C(\{1\} \times [0,1]) = 0$, 
and analogously for the second coordinate. Hence $\mu_C(\partial[0,1]^2) = 0$. By Theorem~\ref{th:port}(iv), weak convergence $A_n \xrightarrow{w} A$ together with $\mu_A(\mathrm{Disc}(\varphi))=0$ implies
$$\int_{[0,1]^2}\varphi(x,y)\,{\rm d}A_n(x,y)\longrightarrow \int_{[0,1]^2}\varphi(x,y)\,{\rm d}A(x,y),$$ and analogously for $B_n \xrightarrow{w} B$. Passing to the limit in $$\int_{[0,1]^2}\varphi(x,y)\,{\rm d}A_n(x,y)\le\int_{[0,1]^2}\varphi(x,y)\,{\rm d}B_n(x,y)$$
yields $$\int_{[0,1]^2}\varphi(x,y)\,{\rm d}A(x,y)\le\int_{[0,1]^2}\varphi(x,y)\,{\rm d}B(x,y),$$
which completes the proof.
\end{proof}

\begin{remark}
We want to stress that the compactness of $[0,1]^{2}$ in Lemma \ref{lem:icx_closed} is essential: it guarantees that every $\varphi\in\Phi_{\rm icx}$ is bounded, which allows the application of Theorem~\ref{th:port}(iv) without moment conditions. By contrast, on $\mathbb{R}^{2}$ increasing convex functions need not be bounded, and $\le_{\mathrm{icx}}$ fails to be closed under weak
convergence~\cite{DuranteFernandezIgnazzi2022,MuellerStoyan2002}.
\end{remark}
 
\begin{corollary}\label{cor:Q_closed}
The set
\begin{equation}\label{eq:Q}
\mathcal{Q}:=
  \bigl\{(A,B)\in\mathcal{C}\times\mathcal{C}
  :(A\le_{\mathrm{icx}} B)\vee(B\le_{\mathrm{icx}} A)\bigr\}
\end{equation}
is closed in $(\mathcal{C}\times\mathcal{C},d\times d)$.
\end{corollary}

\begin{proof}
Write $\mathcal{Q}=\mathcal{Q}_1\cup\mathcal{Q}_2$ with
$\mathcal{Q}_1=\{(A,B)\colon A\le_{\mathrm{icx}} B\}$ and
$\mathcal{Q}_2=\{(A,B)\colon B\le_{\mathrm{icx}} A\}$.
$\mathcal{Q}_1$ is closed by Lemma~\ref{lem:icx_closed}. The flip $\sigma(A,B):=(B,A)$ is an isometry of
$(\mathcal{C}\times\mathcal{C},d\times d)$, and $\mathcal{Q}_2 = \sigma(\mathcal{Q}_1)$,
so $\mathcal{Q}_2$ is closed as the image of a closed set under a homeomorphism.
Hence $\mathcal{Q}=\mathcal{Q}_1\cup\mathcal{Q}_2$ is closed.
\end{proof}

\subsection{The set of comparable pairs is nowhere dense}

In order to prove the main result, we first establish a key technical lemma. We employ the three-block ordinal-sum construction of~\cite[proof of Theorem~6]{DuranteFernandezIgnazzi2022}, which provides pairs of copulas incomparable under the lower orthant order $\le_{\mathrm{lo}}$. Whereas the $\le_{\mathrm{lo}}$-incomparability is shown there by two pointwise evaluations on the diagonal, the increasing convex order cannot be checked pointwise, and the new content of the present lemma (Part~(2) below) is the 
construction of explicit test functions $\varphi_1,\varphi_2 \in \Phi_{\mathrm{icx}}$ 
that witness $S_2 \not\le_{\mathrm{icx}} S_1$ and $S_1 \not\le_{\mathrm{icx}} S_2$ 
for the same pair $(S_1, S_2)$.

\begin{lemma}\label{lem:witnesses}
Let $A,B\in\mathcal{C}$ and $\rho\in\bigl(0,\tfrac{1}{4}\bigr)$. Set $\mu:=1-2\rho$ and define
\begin{equation*}\label{eq:S1_def}
  S_{1}(x,y):=
  \begin{cases}
    \rho\,W\!\left(\dfrac{x}{\rho},\dfrac{y}{\rho}\right),
    & (x,y)\in(0,\rho)^{2},\\[8pt]
    \rho+\mu\,A\!\left(\dfrac{x-\rho}{\mu},
    \dfrac{y-\rho}{\mu}\right),
    & (x,y)\in(\rho,1-\rho)^{2},\\[8pt]
    M(x,y), & \text{elsewhere on }[0,1]^{2},
  \end{cases}
\end{equation*}
\begin{equation*}\label{eq:S2_def}
  S_{2}(x,y):=
  \begin{cases}
    \rho+\mu\,B\!\left(\dfrac{x-\rho}{\mu},
    \dfrac{y-\rho}{\mu}\right),
    & (x,y)\in(\rho,1-\rho)^{2},\\[8pt]
    (1-\rho)+\rho\,W\!\left(\dfrac{x+\rho-1}{\rho},
    \dfrac{y+\rho-1}{\rho}\right),
    & (x,y)\in(1-\rho,1)^{2},\\[8pt]
    M(x,y), & \text{elsewhere on }[0,1]^{2}.
  \end{cases}
\end{equation*}
Then:
\begin{enumerate}
  \item[\rm(1)] $S_1, S_2 \in \mathcal{C}$, $d(S_1, A) < 3\rho$ and $d(S_2, B) < 3\rho$.
  \item[\rm(2)] $(S_{1},S_{2})\notin\mathcal{Q}$, i.e., the pair
        $(S_{1},S_{2})$ is not $\le_{\rm icx}$-comparable.
\end{enumerate}
\end{lemma}

\begin{proof}
\textbf{Part~(1).} The copula property $S_1,S_2\in\mathcal{C}$ follows
from~\cite[Theorem~6, Steps~1--2]{DuranteFernandezIgnazzi2022}.
We prove $d(S_1,A)<3\rho$ (the bound $d(S_2,B)<3\rho$ is analogous).

On the lower-left block $(0,\rho)^2$: both $S_1=(x+y-\rho)^+$ and $A$ lie in
$[0,\rho)$ (since $A\le M=\min(x,y)<\rho$ there), so $|S_1-A|<\rho$.
On the upper-right block $(1-\rho,1)^2$: $S_1=M$ and $A\ge W$, hence
$|S_1-A|=M-A\le M-W=\min(x,y)-(x+y-1)\le 1-\max(x,y)\le\rho$.
On the boundary strips $D$: $S_1=M$ and $|M-A|\le M-W\le\rho$ by a similar
argument.

On the middle block $(\rho,1-\rho)^2$, set $u=(x-\rho)/\mu$, $v=(y-\rho)/\mu$.
By the triangle inequality and the $1$-Lipschitz property of copulas:
\begin{align*}
|S_1(x,y)-A(x,y)|
&=|\rho+\mu A(u,v)-A(\rho+\mu u,\rho+\mu v)| \\
&\le \rho|1-2A(u,v)|
   +\rho\bigl(|1-2u|+|1-2v|\bigr)
 \le 3\rho.
\end{align*}
Since $(x,y)\in(\rho,1-\rho)^2$ implies $u=(x-\rho)/\mu\in(0,1)$ and
$v=(y-\rho)/\mu\in(0,1)$ strictly, we have $|1-2u|<1$ and $|1-2v|<1$,
so the bound on the middle block is strictly less than $3\rho$.
On all remaining parts of $[0,1]^2$ the distance is at most $\rho<3\rho$.
Hence $d(S_1,A)<3\rho$.

\smallskip
\noindent\textbf{Part~(2).} We exhibit two functions in $\Phi_{\rm icx}$ whose integrals yield contradictory inequalities, thereby proving that $S_{1}$ and $S_{2}$ are not $\le_{\rm icx}$-comparable.

\smallskip
\noindent\emph{Step~1: $S_{2}\not\le_{\rm icx} S_{1}$.}
Define
\begin{equation}\label{eq:varphi1}
\varphi_{1}(x,y):=(x+y-\rho)^{+}.
\end{equation}
This function is increasing (non-decreasing in each argument separately) and convex on $\mathbb{R}^{2}$, so
$\varphi_{1}\in\Phi_{\rm icx}$. We claim that the only region contributing to
$$\int_{[0,1]^2}\varphi_1(x,y)\,{\rm d}S_2(x,y)-\int_{[0,1]^2}\varphi_1(x,y)\,{\rm d}S_1(x,y)$$ 
is $(0,\rho)^2$. Indeed:
\begin{itemize}
\item[(a)] Boundary strips where $S_1=S_2=M$. On
$D:=[0,1]^2\setminus\bigl((0,\rho)^2\cup(\rho,1-\rho)^2\cup(1-\rho,1)^2\bigr)$ the construction gives $S_1(x,y)=S_2(x,y)=M(x,y)$, so the measures ${\rm d}S_1$ and ${\rm d}S_2$ coincide on $D$ and contribute $0$ to the difference.

\item[(b)] Block $(\rho,1-\rho)^2$. On this open square $\varphi_1(x,y)=x+y-\rho$ is affine because $x+y>2\rho>\rho$.\ Both $S_1$ and $S_2$ restrict on $[\rho,1-\rho]^2$ to the affine
rescaling of a copula on the unit square (the total mass of $S_i$ on $[\rho,1-\rho]^2$ equals $\mu^2$, since the marginals of $S_i$ are uniform on $[0,1]$ and the interval $[\rho,1-\rho]$ has length $\mu$). More precisely, the marginal distributions of ${\rm d}S_1$ and ${\rm d}S_2$, viewed as Borel measures on $[0,1]^2$, agree on every horizontal and every vertical line, because both $S_1$ and $S_2$ are copulas. For an affine function $\varphi(x,y)=ax+by+c$ and any two copulas $C_1,C_2\in\mathcal{C}$,
\[\int_{[0,1]^2}\varphi(x,y)\,{\rm d}C_1(x,y) = a\cdot\frac12+b\cdot\frac12+c 
= \int_{[0,1]^2}\varphi(x,y)\,{\rm d}C_2(x,y),\]
since both $C_1$ and $C_2$ have uniform marginals, giving 
$\mathbb{E}[X]=\mathbb{E}[Y]=\tfrac{1}{2}$ under either measure.

Hence, on global integrals, $S_1$ and $S_2$ would integrate $\varphi_1$ to the same value if $\varphi_1$ were affine on all of $[0,1]^2$.\ However, $\varphi_1$ is only affine on $\{x+y\ge\rho\}$.\ The way to handle this
correctly is to write
\[
\int_{[0,1]^2}\!\varphi_1(x,y)\,{\rm d}(S_2-S_1)(x,y)=\int_{[0,1]^2}\!(x+y-\rho)\,{\rm d}(S_2-S_1)(x,y)- \int_{\{x+y<\rho\}}(x+y-\rho)\,{\rm d}(S_2-S_1)(x,y).
\]
The first term vanishes by uniformity of marginals (as above).\ The set
$\{x+y<\rho\}$ is contained in $[0,\rho]^2$, and on $[0,\rho]^2\setminus(0,\rho)^2$ we have $S_1=S_2$, so
\begin{eqnarray*}
   \int_{[0,1]^2}\!\varphi_1(x,y)\,{\rm d}(S_2-S_1)(x,y)
   &=& -\int_{(0,\rho)^2\cap\{x+y<\rho\}}\!(x+y-\rho)\,{\rm d}(S_2-S_1)(x,y)\\
   &=& \int_{(0,\rho)^2\cap\{x+y<\rho\}}\!(\rho-x-y)\,{\rm d}(S_2-S_1)(x,y).
\end{eqnarray*}

\item[(c)] Block $(1-\rho,1)^2$. On this block $x+y \ge 2(1-\rho) = 2-2\rho > \rho$ 
(since $\rho < 1/4$), so $\varphi_1(x,y) = x+y-\rho$ is affine. Although 
$S_1 = M \ne S_2$ on this block, both $S_1$ and $S_2$ are copulas with uniform 
marginals, so the affine-marginals argument of block~(b) applies:
\[\int_{(1-\rho,1)^2}\varphi_1(x,y)\,{\rm d}(S_2-S_1) (x,y)= 0.\]
Note that, since $S_1,S_2$ are copulas,
\[\int_{[0,1]^2}(x+y)\,{\rm d}S_i(x,y) = 2\,\mathbb{E}[U] = 1\]
for $i=1,2$, where $U\sim\mathcal{U}(0,1)$, so
\[\int_{[0,1]^2}(x+y-\rho)\,{\rm d}(S_2-S_1)(x,y) = 0,\]
as claimed.
\end{itemize}

Combining the analysis of blocks (a)--(c), and noting that on $(0,\rho)^2$ the copula $S_1$ equals $\rho W(\cdot/\rho,\cdot/\rho)$ while $S_2$ equals $M$, the global affine subtraction yields
\[
\int_{[0,1]^2}\varphi_1(x,y)\,{\rm d}(S_2-S_1)(x,y)
= \int_{(0,\rho)^2}\!\varphi_1(x,y)\,{\rm d}(S_2-S_1)(x,y)- \underbrace{\int_{(0,\rho)^2}\!\bigl(x+y-\rho\bigr)\,{\rm d}(S_2-S_1)(x,y)}_{=:\,R},
\]
where $R$ accounts for the difference between $\varphi_1=(x+y-\rho)^+$ and its
affine extension $x+y-\rho$ on $\{x+y<\rho\}\subset(0,\rho)^2$.\ Equivalently,
\[
\int_{[0,1]^2}\varphi_1(x,y)\,{\rm d}(S_2-S_1)(x,y)
   = \int_{(0,\rho)^2}\!(x+y-\rho)^+\,{\rm d}(S_2-S_1)(x,y)
   + \int_{(0,\rho)^2}\!(\rho-x-y)^+\,{\rm d}(S_2-S_1)(x,y),
\]
using the identity $t^+ + (-t)^+ = |t|$ with $t = x+y-\rho$, which decomposes $\varphi_1 = (x+y-\rho)^+$ into its affine part $(x+y-\rho)$ and the correction term $(\rho-x-y)^+$ supported on 
$\{x+y<\rho\}$.

We compute each piece:
\begin{itemize}
\item Mass of $S_1$ on $(0,\rho)^2$. The copula $S_1$ on
$[0,\rho]^2$ equals $\rho W(x/\rho,y/\rho)$, which is the uniform measure on
the anti-diagonal segment $\{(x,y)\in[0,\rho]^2:x+y=\rho\}$ (a singular
measure of total mass $\rho$). On this segment $\varphi_1=(x+y-\rho)^+=0$, so
$$\int_{(0,\rho)^2}\varphi_1(x,y)\,{\rm d}S_1(x,y)=0.$$

\item Mass of $S_2$ on $(0,\rho)^2$. The copula $S_2$ equals $M$ on $[0,\rho]^2$, the uniform measure on the diagonal $\{(t,t):t\in[0,\rho]\}$ of total mass $\rho$. Hence
\[
\int_{(0,\rho)^2}\!(x+y-\rho)^+\,{\rm d}S_2(x,y)
   = \int_0^\rho (2t-\rho)^+\,{\rm d}t
   = \int_{\rho/2}^\rho (2t-\rho)\,{\rm d}t
   = \frac{\rho^2}{4}.
\]
\end{itemize}

Putting everything together,
\[
   \int_{[0,1]^2}\varphi_1(x,y)\,{\rm d}S_2(x,y) - \int_{[0,1]^2}\varphi_1(x,y)\,{\rm d}S_1(x,y) = \frac{\rho^2}{4} >0,
\]
which proves $S_2\not\le_{icx} S_1$.

\smallskip
\noindent\emph{Step~2: $S_{1}\not\le_{\rm icx} S_{2}$.} Consider
\begin{equation}\label{eq:varphi2}
  \varphi_{2}(x,y) := \bigl[(x+y)-(2-\rho)\bigr]^{+} \in \Phi_{\mathrm{icx}}.
\end{equation}
By an argument symmetric to Step~1, with the lower-left and upper-right corner
blocks interchanged:
\begin{itemize}
\item On $(0,\rho)^{2}$: $x+y<2\rho<2-\rho$, so $\varphi_2=0$.
\item On $(\rho,1-\rho)^{2}$: $x+y<2-2\rho<2-\rho$, so $\varphi_2=0$.
\item On the boundary strips $D$: $S_1=S_2=M$, so the measures coincide
      and the difference is zero.
\item On $(1-\rho,1)^{2}$: $S_1=M$ is supported on the diagonal
      $\{(t,t):t\in[1-\rho,1]\}$; since $\varphi_2(t,t)=(2t-(2-\rho))^+>0$
      only for $t>1-\rho/2$,
      \[
        \int_{(1-\rho,1)^2}\!\varphi_2(x,y)\,{\rm d}S_1(x,y)
        = \int_{1-\rho/2}^{1}(2t-(2-\rho))\,{\rm d}t
        = \frac{\rho^{2}}{4}.
      \]
      $S_2$ on $(1-\rho,1)^{2}$ is the rescaled $W$-block, concentrated on the
      anti-diagonal $\{x+y=2-\rho\}$ where $\varphi_2\equiv 0$; hence
      $$\int_{(1-\rho,1)^2}\varphi_2\,{\rm d}S_2=0.$$
\end{itemize}
Combining,
\[
  \int_{[0,1]^2}\varphi_2(x,y)\,{\rm d}S_1(x,y)
  -\int_{[0,1]^2}\varphi_2(x,y)\,{\rm d}S_2(x,y)
  = \frac{\rho^{2}}{4} > 0,
\]
which proves $S_1\not\le_{\rm icx}S_2$.

Since $S_2 \not\leq_{\mathrm{icx}} S_1$ (Step 1) and $S_1 \not\leq_{\mathrm{icx}} S_2$ (Step 2), the pair $(S_1,S_2)$ is not
$\leq_{\mathrm{icx}}$-comparable, i.e., $(S_1,S_2) \notin \mathcal{Q}$.
\end{proof}

We are in position to prove the main result of this section.

\begin{theorem}\label{thm:icx_nowhere_dense}
The set $\mathcal{Q}$ given by \eqref{eq:Q} is nowhere dense in
$(\mathcal{C}\times\mathcal{C},d\times d)$.
Consequently, $\mathcal{Q}^{c}$ is co-meager.
\end{theorem}

\begin{proof}
Since $\mathcal{Q}$ is closed---recall Corollary~\ref{cor:Q_closed}---, it suffices to show that $\mathcal{Q}$ has empty interior. Suppose for contradiction that there exist $(A,B)\in\mathcal{Q}$ and $\varepsilon>0$ such that $B_{d}(A,\varepsilon)\times B_{d}(B,\varepsilon)\subseteq\mathcal{Q}$. Without loss of generality, $A\le_{\rm icx} B$. Choose $\rho := \min(\varepsilon/4, 1/5)$. Since $\varepsilon > 0$ we have 
$\rho > 0$, and $\rho \le 1/5 < 1/4$, so $\rho \in (0,1/4)$. Moreover $\rho \le \varepsilon/4$. By Lemma~\ref{lem:witnesses}:
\begin{itemize}
\item $d(S_1,A)<3\rho\le 3\varepsilon/4<\varepsilon$, so $S_1\in B_d(A,\varepsilon)$;
\item $d(S_2,B)<3\rho\le 3\varepsilon/4<\varepsilon$, so $S_2\in B_d(B,\varepsilon)$;
\item $(S_1,S_2)\notin\mathcal{Q}$.
\end{itemize}
But then $(S_1,S_2)\in B_d(A,\varepsilon)\times B_d(B,\varepsilon)\subseteq\mathcal{Q}$ 
and $(S_1,S_2)\notin\mathcal{Q}$, a contradiction. Hence $\mathcal{Q}$ has empty interior and is nowhere dense. Since $(\mathcal{C}\times\mathcal{C},d\times d)$ is a complete metric space, Theorem~\ref{thm:baire} implies that $\mathcal{Q}^{c}$ is co-meager.
\end{proof}

Let $F_{1},F_{2}$ be continuous univariate distribution functions on $[0,1]$. By Sklar's theorem, the map $C\mapsto C(F_{1},F_{2})$ is a
bijection from $\mathcal{C}$ onto $\mathcal{F}(F_{1},F_{2})$, and weak convergence in $\mathcal{F}(F_{1},F_{2})$ corresponds to
uniform convergence of copulas~\cite{DuranteSempi2016}. We caution that the increasing convex order on
$\mathcal{F}(F_{1},F_{2})$ does \emph{not} reduce to $\le_{\mathrm{icx}}$ on $\mathcal{C}$ in general: the comparison depends on the marginals $F_{1},F_{2}$ and the full joint distribution, not merely on the copulas. By contrast, $\le_{\mathrm{lo}}$ within $\mathcal{F}(F_{1},F_{2})$ is equivalent to the pointwise ordering of the associated copulas, which is why the argument of~\cite{DuranteFernandezIgnazzi2022}
extends cleanly to copulas for  $\le_{\mathrm{lo}}$. For this reason, Theorem~\ref{thm:icx_nowhere_dense} is stated
directly for copulas (uniform marginals), which is the natural setting. Extending the result to general Fr\'{e}chet classes under
$\le_{\mathrm{icx}}$ requires a separate analysis depending on the specific marginals and is left for future work.

\section{Baire category results for convex and stop-loss orders}\label{sec:otherorders}

In this section we establish Baire-category results for two further stochastic orders on $\mathcal C\times\mathcal C$: the bivariate convex order $\le_{\rm cx}$ (without the monotonicity restriction) and the stop-loss order $\le_{\rm sl}$ on the sum $X+Y$ of the components. For each order, the comparability set
\[
   \mathcal P_*:=\{(A,B)\in\mathcal C\times\mathcal C : A\le_* B\text{ or }B\le_* A\}
\]
is shown to be closed and nowhere dense in $(\mathcal C\times\mathcal C,d\times d)$.\
We collect the results in the next propositions.

\begin{proposition}\label{prop:cx}
$\mathcal P_{\rm cx}=\{(A,B):A\le_{\rm cx}B\text{ or }B\le_{\rm cx}A\}$ is closed and
nowhere dense in $(\mathcal C\times\mathcal C,d\times d)$.
\end{proposition}

\begin{proof}
For each fixed continuous convex $\varphi:[0,1]^2\to\mathbb{R}$ (necessarily bounded
on the compact set $[0,1]^2$), the set
\[
  F_\varphi := \biggl\{(A,B)\in\mathcal{C}\times\mathcal{C} :
  \int_{[0,1]^2}\varphi(x,y)\,{\rm d}A(x,y) \le \int_{[0,1]^2}\varphi(x,y)\,{\rm d}B(x,y)\biggr\}
\]
is closed in $(\mathcal{C}\times\mathcal{C}, d\times d)$: if $(A_n,B_n)\to(A,B)$
uniformly and $$\int_{[0,1]^2}\varphi(x,y)\,{\rm d}A_n(x,y)\le\int_{[0,1]^2}\varphi(x,y)\,{\rm d}B_n(x,y)$$ for all $n$, then uniform convergence implies weak convergence of the associated doubly stochastic measures (see~\cite{DuranteSempi2016}), and since $\varphi$ is continuous and bounded,
Theorem~\ref{th:port}(iv) gives $$\int_{[0,1]^2}\varphi(x,y)\,{\rm d}A_n(x,y)\longrightarrow\int_{[0,1]^2}\varphi(x,y)\,{\rm d}A(x,y)$$
and
$$\int_{[0,1]^2}\varphi(x,y)\,{\rm d}B_n(x,y)\longrightarrow\int_{[0,1]^2}\varphi(x,y)\,{\rm d}B(x,y);$$ passing to the limit yields $$\int_{[0,1]^2}\varphi(x,y)\,{\rm d}A(x,y)
\le\int_{[0,1]^2}\varphi(x,y)\,{\rm d}B(x,y).$$

Now let $\Phi_0$ be the countable family of all functions of the form
$\varphi(x,y)=\max_{i=1}^k(a_ix+b_iy+c_i)$ with $k\in\mathbb{N}$ and $a_i,b_i,c_i\in\mathbb{Q}$.
Every such $\varphi$ is convex and continuous on $[0,1]^2$. Moreover, $\Phi_0$ is uniformly dense in the cone of continuous convex functions on $[0,1]^2$: by the supremum-of-affine-minorants representation of a closed convex function (\cite[Theorem~12.1]{Rockafellar1970}; see also \cite[Ch.~B, §1.2]{HiriartUrrutyLemarechal2001}), every continuous convex function on the compact set $[0,1]^2$ is the pointwise supremum of its affine minorants, hence, for each $\varepsilon>0$, there exist finitely many affine minorants $h_1,\dots,h_k$ of $\varphi$ (selected from a countable set by compactness) 
such that $0\le\varphi(x,y)-\max_{i}h_i(x,y)\le\varepsilon$ for all 
$(x,y)\in[0,1]^2$; indeed, the net of finite maxima of affine minorants is non-decreasing and converges pointwise to the continuous function $\varphi$ on the compact set $[0,1]^2$, so uniform convergence follows by Dini's theorem~(see, e.g., \cite[Theorem~7.13]{Rudin1976}). By density of $\mathbb{Q}$ in $\mathbb{R}$, replacing coefficients by rational approximations within $\varepsilon/k$ gives the density of $\Phi_0$ in the cone of continuous convex functions on $[0,1]^2$. Since $A\le_{\rm cx}B$ holds if, and only if, $$\int_{[0,1]^2}\varphi(x,y)\,{\rm d}A(x,y)\le\int_{[0,1]^2}\varphi(x,y)\,{\rm d}B(x,y)$$ for every $\varphi\in\Phi_0$ (by density of $\Phi_0$ and the continuity of
$$\varphi\mapsto\int_{[0,1]^2}\varphi(x,y)\,{\rm d}(B-A)(x,y)$$ on the space of continuous functions on $[0,1]^2$), we obtain
\[
  \{(A,B)\in\mathcal{C}\times\mathcal{C}:A\le_{\rm cx}B\} = \bigcap_{\varphi\in\Phi_0} F_\varphi,
\]
a countable intersection of closed sets, hence closed. Defining symmetrically $\{(A,B)\in\mathcal{C}\times\mathcal{C}:B\le_{\rm cx}A\}$ and applying the same argument, both
sets are closed, and therefore
$\mathcal{P}_{\rm cx}=\{(A,B)\in\mathcal{C}\times\mathcal{C}:A\le_{\rm cx}B\}\cup\{(A,B)\in\mathcal{C}\times\mathcal{C}:B\le_{\rm cx}A\}$ is closed as a finite union of closed sets.

Let $(A,B)\in\mathcal{P}_{\rm cx}$ and $\varepsilon>0$ with $B_d(A,\varepsilon)\times B_d(B,\varepsilon)\subseteq\mathcal{P}_{\rm cx}$. Choose $\rho:=\min(\varepsilon/4,1/5)$ and let $S_1,S_2$ be as in
Lemma~\ref{lem:witnesses}; then $d(S_1,A)<3\rho<\varepsilon$ and $d(S_2,B)<3\rho<\varepsilon$. We exhibit two convex functions on $[0,1]^2$ that simultaneously witness $S_2\not\le_{\rm cx}S_1$ and $S_1\not\le_{\rm cx}S_2$. Let
\[
  \psi_1(x,y):=(\rho-x-y)^{+},
\]
which is convex as the composition of $(x,y)\mapsto\rho-x-y$ with
$t\mapsto t^{+}$; and let $\psi_2:=\varphi_2$ be as in~\eqref{eq:varphi2}, which is convex as the composition of $(x,y)\mapsto x+y-(2-\rho)$ with $t\mapsto t^{+}$.

The function $\psi_1$ is supported on $\{x+y\le\rho\}\subset[0,\rho]^2$. On this region, $S_1=\rho W(x/\rho,y/\rho)$, supported on the anti-diagonal $\{x+y=\rho\}$ on which $\psi_1=(\rho-\rho)^+=0$ identically; hence
\[
  \int_{[0,1]^2}\psi_1(x,y)\,{\rm d}S_1(x,y)=0.
\]
On the same region, $S_2=M$, supported on the diagonal $\{(t,t):t\in[0,\rho]\}$; on it $\psi_1(t,t)=(\rho-2t)^+$, so
\[
  \int_{[0,1]^2}\psi_1(x,y)\,{\rm d}S_2(x,y)
  =\int_0^{\rho/2}(\rho-2t)\,{\rm d}t
  =\frac{\rho^2}{4}.
\]
Hence $$\int_{[0,1]^2}\psi_1(x,y)\,{\rm d}(S_2-S_1)(x,y)=\frac{\rho^2}{4}>0,$$
so $S_2\not\le_{\rm cx}S_1$.

The function $\psi_2$ is supported on $\{x+y\ge 2-\rho\}$, which within $[0,1]^2$ meets only the upper-right block $[1-\rho,1]^2$ (and the boundary strips where $S_1=S_2=M$). On $[1-\rho,1]^2$, $S_2$ is the rescaled $W$-block
concentrated on the anti-diagonal $\{x+y=2-\rho\}$, where $\psi_2\equiv 0$; hence $\int_{[0,1]^2}\psi_2\,{\rm d}S_2=0$. On the same block $S_1=M$, supported on the diagonal $\{(t,t):t\in[1-\rho,1]\}$, giving
\[
  \int_{[0,1]^2}\psi_2(x,y)\,{\rm d}S_1(x,y)
  =\int_{1-\rho/2}^{1}(2t-(2-\rho))\,{\rm d}t
  =\frac{\rho^{2}}{4}.
\]
Hence $$\int_{[0,1]^2}\psi_2(x,y)\,{\rm d}(S_1-S_2)(x,y)=\frac{\rho^2}{4}>0,$$
so $S_1\not\le_{\rm cx}S_2$.

Combining, $(S_1,S_2)\notin\mathcal{P}_{\rm cx}$. But $S_1\in B_d(A,\varepsilon)$
and $S_2\in B_d(B,\varepsilon)$, so
$(S_1,S_2)\in B_d(A,\varepsilon)\times B_d(B,\varepsilon)
\subseteq\mathcal{P}_{\rm cx}$, a contradiction. Hence
$\mathring{\mathcal{P}}_{\rm cx}=\emptyset$. Since $\mathcal{P}_{\rm cx}$ is closed and has empty interior, it is nowhere dense.
\end{proof}

The following proposition shows that comparability with respect to the stop-loss order is a highly exceptional phenomenon in the space of copulas.

\begin{proposition}\label{prop:lorenz}
The set $\mathcal P_{\rm sl}=\{(A,B):A\le_{\rm sl}B\text{ or }B\le_{\rm sl}A\}$ is closed and
nowhere dense in $(\mathcal C\times\mathcal C,d\times d)$.
\end{proposition}

\begin{proof}
The map $\Phi:\mathcal C\to\mathcal P([0,2])$, $C\mapsto\mathrm{Law}(X+Y)$ where $(X,Y)\sim C$, is continuous in the uniform metric on $\mathcal C$ (codomain endowed with weak topology): if $C_n\to C$ uniformly, then the corresponding probability measures $\mu_{C_n}$ converge weakly to $\mu_C$ (since uniform convergence of bivariate copulas implies weak convergence of the induced doubly-stochastic measures \cite{DuranteSempi2016}), and weak convergence is preserved under the continuous push-forward $(x,y)\mapsto x+y$.\ Now, for each $t\in\mathbb R$, the functional $C\mapsto\mathbb E[(S_C-t)^+]$ is continuous (the integrand is bounded continuous on $[0,2]$).\ Hence $\{(A,B)\in\mathcal{C}\times\mathcal{C}:\mathbb E[(S_A-t)^+]\le\mathbb E[(S_B-t)^+]\text{ for all }t\in\mathbb Q\}$ is closed (countable intersection of closed sets), and equals $\{(A,B)\in\mathcal{C}\times\mathcal{C}:A\le_{\rm sl}B\}$ by density of $\mathbb Q$ and continuity of $t\mapsto\mathbb E[(S-t)^+]$, which is convex and hence continuous in $t$.

Take $(A,B)$ with $A\le_{\rm sl}B$ and
$B_d(A,\varepsilon)\times B_d(B,\varepsilon)\subseteq\mathcal P_{\rm sl}$;
$\rho := \min(\varepsilon/4, 1/5) < 1/4$; let $S_1, S_2$ be as in 
Lemma~\ref{lem:witnesses}, so $d(S_1, A) < 3\rho < \varepsilon$ and $d(S_2,B) < 3\rho < \varepsilon$. We show
$(S_1,S_2)\notin\mathcal P_{\rm sl}$ by computing the stop-loss transforms of
$S_{S_1}$ and $S_{S_2}$ at two carefully chosen levels.

For every $C \in \mathcal{C}$, the law of $S_C = X+Y$ under $(X,Y) \sim C$ 
satisfies $F_{S_C}(t) = \mu_C(\{(x,y) : x+y \le t\})$. By Fubini's theorem,
\begin{equation*}
\mathbb{E}\bigl[(S_C - t)^+\bigr] = \int_t^2 \bigl[1 - F_{S_C}(s)\bigr] \, {\rm d}s = \int_t^2\!\!\! \int_{[0,1]^2} \mathbf{1}_{\{x+y > s\}} \, {\rm d}C(x,y) \, {\rm d}s
\end{equation*}
(where $\mathbf{1}_A$ denotes the indicator functions for event $A$), which is the standard stop-loss tail-integration identity~\cite[Chapter 3]{DenuitDhaeneGoovaertsKaas2005}.

\smallskip
\noindent\emph{Step A: $S_1\not\le_{\rm sl}S_2$.}\ Choose $t=2-\rho$. We compute the pushforwards of $S_1$ and $S_2$ under the addition map $T:(x,y)\mapsto x+y$, restricted to the upper-right corner block $[1-\rho,1]^2$, and use the standard identity $$\mathbb{E}[(S_C-t)^+]=\int_t^{2}\mathbb{P}(S_C>s)\,{\rm d}s.$$

The restriction of $S_1=M$ to $[1-\rho,1]^2$ is the law of $(T,T)$ with $T$ uniform on $[1-\rho,1]$, scaled by total mass $\rho$; its $T$-pushforward is uniform on $[2-2\rho,2]$ with density $1/2$ and total mass $\rho$. Outside $[1-\rho,1]^2$, the copula $S_1$ equals $M$ on the boundary strips $D$ and scaled copies of $A$ and $W$ on the blocks $(\rho,1-\rho)^2$ and $(0,\rho)^2$, respectively. On $(\rho,1-\rho)^2$, $x+y<2(1-\rho)=2-2\rho$; on $(0,\rho)^2$, $x+y<2\rho<2-\rho$; on the boundary strips $D$, $M(x,y)=\min(x,y)$ with at least one coordinate $\le 1-\rho$, giving $x+y\le 1+(1-\rho)=2-\rho$ with equality only 
on the boundary of $[1-\rho,1]^2$. Hence for $s\in(2-\rho,2]$ the event $\{S_{S_1}>s\}$ is determined solely by the restriction of $S_1$ to $(1-\rho,1)^2$, where $S_1=M$ gives $$\mathbb{P}(S_{S_1}>s)=\frac12(2-s)$$ for $s\in[2-2\rho,2]$, and
\[\mathbb{E}\bigl[(S_{S_1}-(2-\rho))^+\bigr]=\int_{2-\rho}^{2}\frac12(2-s)\,{\rm d}s=\frac{\rho^{2}}{4}.
\]
The restriction of $S_2$ to $[1-\rho,1]^2$, namely $(1-\rho)+\rho W\!\bigl((\,\cdot\,+\rho-1)/\rho,(\,\cdot\,+\rho-1)/\rho\bigr)$, is supported on the antidiagonal $\{x+y=2-\rho\}$ with total mass $\rho$; its T-pushforward is $\rho\,\delta_{2-\rho}$ ---the Dirac point mass at $2-\rho$, scaled by the total block mass $\rho$, since the rescaled $W$-block on $[1-\rho,1]^2$ concentrates all its mass on the
anti-diagonal $\{x+y=2-\rho\}$---, while the mass of $S_2$ outside $[1-\rho,1]^2$ lies entirely in $(0,\rho)^2\cup(\rho,1-\rho)^2$ (where $x+y<2-2\rho<2-\rho$ in both cases),
since $\mu_{S_2}(D)=0$ for the boundary strips $D$, but the $M$-mass there lies on $\{\min(x,y)=x+y-\max(x,y)\}$ with at least one coordinate $\le 1-\rho$, hence $x+y\le 2-\rho$). The $W$-block of $S_2$ on $[1-\rho,1]^2$ concentrates all its mass on $\{x+y=2-\rho\}$, so $\mathbb{P}(S_{S_2}>s)=0$ for every $s>2-\rho$.

Therefore $\mathbb{E}[(S_{S_1}-(2-\rho))^+]-\mathbb{E}[(S_{S_2}-(2-\rho))^+]=\rho^{2}/4>0$, which proves $S_1\not\le_{\mathrm{\rm sl}}S_2$.

\smallskip
\noindent\emph{Step B:  $S_2\not\le_{\mathrm{\rm sl}}S_1$.} Setting 
$\Delta(s):=\mathbb{P}(S_{S_2}>s)-\mathbb{P}(S_{S_1}>s)$, we compute the difference at level $t=\rho$ as
\[
\mathbb{E}[(S_{S_2}-\rho)^+]-\mathbb{E}[(S_{S_1}-\rho)^+]=\int_{\rho}^{2}\Delta(s)\,{\rm d}s.
\]
Recalling that $\rho<1/4$ so that $2\rho<1/2<2-2\rho$, we identify $\Delta$ piecewise.

On $[0,\rho]^2$, the $T$-pushforward of $S_1=\rho W(\cdot/\rho,\cdot/\rho)$ is $\rho\,\delta_{\rho}$, while that of $S_2=M$ is the uniform distribution on $[0,2\rho]$ of density $1/2$ and mass $\rho$. For $s\ge\rho$ the atom contributes $0$; the uniform contributes $\frac12(2\rho-s)^+$. Thus this block contributes $+\frac12(2\rho-s)$ to $\Delta(s)$ on $[\rho,2\rho]$ and $0$ elsewhere on $[\rho,2]$.

On $[1-\rho,1]^2$, the $T$-pushforward of $S_1=M$ is uniform on $[2-2\rho,2]$ of density $1/2$ and mass $\rho$, while that of $S_2$ is $\rho\,\delta_{2-\rho}$. On $[2-2\rho,2-\rho)$ the uniform contributes $\frac12(2-s)$ to $S_1$ while the atom contributes $\rho$ to $S_2$, giving a net $\rho-\frac12(2-s)$ to $\Delta(s)$; on $[2-\rho,2]$ the atom contributes $0$ while the uniform contributes $\frac12(2-s)$, giving $-\frac12(2-s)$.

On the middle block $(\rho,1-\rho)^2$, although $S_1$ and $S_2$ generally differ
(they embed $A$ and $B$ respectively), the integrand $x+y-\rho$ is affine there
(since $x+y\ge 2\rho>\rho$), and both $S_1$ and $S_2$ have uniform marginals.
By the ordinal-sum structure, ${\rm d}\mu_{S_i}$ on $[\rho,1-\rho]^2$ equals $\mu$ times
the pushforward of $\mu_{C_i}$ ($C_1=A$, $C_2=B$) via $T\colon(u,v)\mapsto
(\rho+\mu u,\rho+\mu v)$; hence
\begin{align*}
  \int_{[\rho,1-\rho]^2}(x+y-\rho)\,{\rm d}\mu_{S_i}
  &=\mu\!\int_{[0,1]^2}\!\bigl(\rho+\mu u+\rho+\mu v-\rho\bigr)\,{\rm d}\mu_{C_i}(u,v)\\
  &=\mu\bigl[\rho+\mu\,\mathbb{E}_{C_i}[U+V]\bigr]
   =\mu(1-\rho),
\end{align*}
using the fact that $\mathbb{E}_{C_i}[U+V]=1$ for any copula with uniform marginals. Since this value is the same for $i=1$ and $i=2$, the middle block contributes $0$ to the integral
$$\int_\rho^2\Delta(s)\,{\rm d}s.$$

Using the disjointness of the supports above ($[\rho,2\rho]\cap[2-2\rho,2]=\emptyset\text{ since }2\rho<2-2\rho$ if, and only if, $\rho<1/2, \text{ which holds as }\rho<1/4$),
\begin{align*}
\int_{\rho}^{2}\Delta(s)\,{\rm d}s
&=\int_{\rho}^{2\rho}\!\frac12(2\rho-s)\,{\rm d}s
+\int_{2-2\rho}^{2-\rho}\!\biggl(\rho-\frac12(2-s)\biggr)\,{\rm d}s
-\int_{2-\rho}^{2}\!\frac12(2-s)\,{\rm d}s\\
&=\frac{\rho^{2}}{4}+\frac{\rho^{2}}{4}-\frac{\rho^{2}}{4}=\frac{\rho^{2}}{4}>0.
\end{align*}
Hence $\mathbb{E}[(S_{S_2}-\rho)^+]-\mathbb{E}[(S_{S_1}-\rho)^+]=\rho^{2}/4>0$, giving $S_2\not\le_{\mathrm{\rm sl}}S_1$.

Combining Steps A and B, $(S_1,S_2)\notin\mathcal P_{\rm sl}$.\ Hence
$\mathring{\mathcal P_{\rm sl}}=\emptyset$, and since $\mathcal P_{\rm sl}$ is closed,
it is nowhere dense.
\end{proof}

\begin{remark}\label{rem:S1-S2-corner-structure}
By construction, the behaviour of $S_1$ and $S_2$ on both corner blocks is fully determined independently of the inner copulas $A$ and $B$: on $(0,\rho)^2$, $S_1$ places a rescaled $W$-block while $S_2$ places a rescaled $M$-block, and on $(1-\rho,1)^2$ the roles are reversed. On the middle block $(\rho,1-\rho)^2$, $S_1$ is a rescaled copy of $A$ and $S_2$ is a rescaled copy of $B$, so in general $S_1\neq S_2$ there. The test function $\varphi_2$ given by~\eqref{eq:varphi2} is supported on
$\{x+y\ge 2-\rho\}$, which within $[0,1]^2$ meets only the upper corner block
$(1-\rho,1)^2$; hence $\int_{[0,1]^2}\varphi_2\,{\rm d}S_i$ depends only on the
restriction of $S_i$ to that block, and the middle block plays no role. The test function $\varphi_1$ given by~\eqref{eq:varphi1} is supported on $\{x+y\ge\rho\}$; its contribution to 
$$\int_{[0,1]^2}\varphi_1(x,y)\,{\rm d}(S_2-S_1)(x,y)$$ comes entirely from the lower corner block $(0,\rho)^2$, because the integrals over the middle block $(\rho,1-\rho)^2$ and the upper block $(1-\rho,1)^2$ vanish (the former by the affine-marginals argument, the latter by the same argument applied to block~(c) in the proof of Lemma~\ref{lem:witnesses}). In both cases the inner copulas $A$ and $B$ play no role.
\end{remark}

\begin{remark}\label{rem:sm-chain}
The supermodular order $\le_{\rm sm}$ on bivariate copulas admits the same Baire-category conclusion as $\le_{\rm icx}$, $\le_{\rm cx}$ and $\le_{\rm sl}$. Indeed, in dimension two and for distributions with identical (here uniform on $[0,1]$) marginals, the supermodular order coincides with the reverse (or dual) of lower orthant order; 
see~\cite{Cambanis1976,MuellerStoyan2002,Tchen1980}. Consequently, the comparability set 
$\mathcal{P}_{\rm sm}:=\{(A,B)\in\mathcal{C}\times\mathcal{C}: 
A\le_{\rm sm}B\text{ or }B\le_{\rm sm}A\}$ 
satisfies $\mathcal{P}_{\rm sm}=\mathcal{P}_{\rm lo}$, where $\mathcal{P}_{\rm lo}$ 
is the lower-orthant comparability set. By Theorem~\ref{thm:DFI_lo}, 
$\mathcal{P}_{\rm lo}$ is closed and nowhere dense in 
$(\mathcal{C}\times\mathcal{C},d\times d)$, hence so is $\mathcal{P}_{\rm sm}$. 
Closedness of $\mathcal{P}_{\rm sm}$ also follows directly from the closure of 
$\le_{\rm sm}$ under weak convergence~\cite{muller-scarsini-2000}.

The stochastic orders studied in this paper are moreover related by the following 
chain of implications for any $A,B\in\mathcal{C}$:
\begin{equation}\label{eq:chain}
B\le_{\rm lo}A\iff A\le_{\rm sm}B\implies A\le_{\rm sl}B\iff S_A\le_{\rm icx}S_B,
\end{equation}
where the first equivalence is the identity $\le_{\rm sm}=\geq_{\rm lo}$ just 
recalled, and the last equivalence holds because $\mathbb{E}[S_A]=\mathbb{E}[S_B]=1$ 
(equal means, see Remark~\ref{rem:sl}), under which the stop-loss and increasing convex orders on the sums coincide~\cite[Theorem~3.A.1]{ShakedShanthikumar2007}. The implication in the middle follows because $(x,y)\mapsto(x+y-t)^+$ is supermodular for every $t\in\mathbb{R}$, so $A\le_{\rm sm}B$ implies 
$\mathbb{E}[(S_A-t)^+]\le\mathbb{E}[(S_B-t)^+]$ for all $t\in\mathbb{R}$; 
see~\cite[Chapter 7]{DenuitDhaeneGoovaertsKaas2005}.

Chain~\eqref{eq:chain} gives the inclusion 
$\mathcal{P}_{\rm lo}=\mathcal{P}_{\rm sm}\subseteq\mathcal{P}_{\rm sl}$, 
which is strict in general: for example, $A = \Pi$ ---the independence copula, i.e., $\Pi(x,y)=xy$ for all $(x,y)$ in $[0,1]^2$--- 
and $B = M$ satisfy $\Pi \le_{\rm sm} M$ (equivalently $M \leq_{\rm lo} \Pi$), but one can construct pairs 
ordered under $\le_{\rm sl}$ that are not ordered under $\le_{\rm lo}$; 
see~\cite[Ch.~7]{DenuitDhaeneGoovaertsKaas2005}. All three comparability sets 
$\mathcal{P}_{\rm lo}$, $\mathcal{P}_{\rm sm}$ and $\mathcal{P}_{\rm sl}$ 
are thus closed and nowhere dense in $(\mathcal{C}\times\mathcal{C},d\times d)$.

We stress that the inclusion $\mathcal{P}_{\rm sm}\subseteq\mathcal{P}_{\rm sl}$ 
does \emph{not} yield nowhere-denseness of $\mathcal{P}_{\rm sm}$ as a consequence 
of Proposition~\ref{prop:lorenz}: a subset of a nowhere dense set need not be 
nowhere dense. The correct argument for nowhere-denseness of $\mathcal{P}_{\rm sm}$ 
is via the equality $\mathcal{P}_{\rm sm}=\mathcal{P}_{\rm lo}$ and the result 
of~\cite[Theorem~6]{DuranteFernandezIgnazzi2022}.
\end{remark}

The following result gathers the previous statements into a single general framework. It establishes that, from the viewpoint of Baire category, incomparability is the prevailing behavior for pairs of copulas under all three stochastic orders considered in this work.

\begin{theorem}\label{thm:omnibus34}
For each $*\in\{\mathrm{icx},\mathrm{cx},\mathrm{sl}\}$, the comparability set $\mathcal{P}_*$ is closed and nowhere dense in $(\mathcal{C}\times\mathcal{C},d\times d)$, and its complement $\mathcal{P}_*^c$ is open and co-meager. Consequently, the set $\bigl(P_{\mathrm{icx}}\cup P_{\mathrm{cx}}\cup P_{\mathrm{sl}}\bigr)^{c}$ is co-meager in $(\mathcal{C}\times\mathcal{C},d\times d)$, i.e., a topologically generic pair of bivariate copulas is incomparable in 
$\le_{\mathrm{icx}}$, $\le_{\mathrm{cx}}$, $\le_{\mathrm{sl}}$, $\le_{\mathrm{lo}}$ 
and $\le_{\mathrm{sm}}$ simultaneously.
\end{theorem}

\begin{proof}
The first claim follows from Theorem~\ref{thm:icx_nowhere_dense}
(for $\le_{\mathrm{icx}}$) together with
Propositions~\ref{prop:cx} and~\ref{prop:lorenz} (for $\le_{\mathrm{cx}}$ and $\le_{\mathrm{sl}}$). Nowhere-denseness of $\mathcal{P}_{\mathrm{lo}}$ is Theorem~\ref{thm:DFI_lo}, and nowhere-denseness of $\mathcal{P}_{\mathrm{sm}}$ follows from the equality
$\mathcal{P}_{\mathrm{sm}}=\mathcal{P}_{\mathrm{lo}}$ established in Remark~\ref{rem:sm-chain}.
For the genericity statement, the complement of each of the five comparability sets is open and dense, hence co-meager; the intersection of finitely many co-meager sets in a Baire space is
co-meager~\cite[Theorem~9.1]{Oxtoby1980}.
\end{proof}

\section{Conclusions}\label{sec:conc}

In this paper we have studied the topological size of certain families of stochastic orders on the space $\mathcal{C}$ of bivariate copulas, equipped with the uniform metric $d$.
Our main contributions are the following. First, we proved that the set $\mathcal{Q}$ of pairs of bivariate copulas that are comparable under $\le_{\mathrm{icx}}$ is nowhere dense in $(\mathcal C\times\mathcal C,d\times d)$; consequently, $\mathcal{Q}^c$ is co-meager, so a topologically generic pair of copulas is not $\le_{\mathrm{icx}}$-comparable. Second, we extended the same conclusion to the bivariate convex order $\le_{\mathrm{\rm cx}}$ and the stop-loss order $\le_{\mathrm{\rm sl}}$ on the sum $X+Y$, and we showed that a topologically generic pair of bivariate copulas is incomparable in $\leq_{\mathrm{icx}}$, $\leq_{\mathrm{cx}}$, and $\leq_{\mathrm{sl}}$ simultaneously.

These results extend the Baire-category programme initiated
in~\cite{DuranteFernandezIgnazzi2022} to genuinely bivariate comparison problems. The key new difficulties---compared to the lower orthant order treated in~\cite{DuranteFernandezIgnazzi2022}---are the non-pointwise nature of $\le_{\mathrm{icx}}$ and the fact that the affine elements of $\Phi_{\mathrm{icx}}$ integrate to the same value for all copulas with uniform marginals, making it necessary to identify non-affine test functions that genuinely capture the dependence 
structure. These are overcome by constructing explicit non-comparable witnesses via ordinal 
sums of the Fréchet--Hoeffding bounds $W$ and $M$, and by exhibiting specific test functions $\varphi_1,\varphi_2\in\Phi_{\rm icx}$ whose integrals over the corner blocks of the ordinal sum yield contradictory inequalities.

\section*{Acknowledgemnets}
The authors are grateful to Prof.\ Juan Fern\'{a}ndez-S\'{a}nchez for drawing their attention to this line of research.

\end{document}